\RequirePackage{fix-cm}
\documentclass[smallextended]{svjour3}       
\smartqed  
\usepackage{graphicx}
\usepackage{graphicx}
\usepackage{amssymb}
\usepackage{amssymb,amsmath,epsfig,multicol}
\usepackage{epstopdf}
\usepackage{fullpage}
\usepackage{subfig}
\usepackage{float}
\usepackage{xcolor}

\newcommand{\R}{\mathbb{R}}

\newcommand{\N}{\mathbb{N}}
\begin{document}
	
	\title{Generalized substantial fractional operators and well-posedness of Cauchy problem 
	}
	
	
	\author{Hafiz Muhammad Fahad         \and
		Mujeeb ur Rehman 
	}
	
	
	\institute{Hafiz Muhammad Fahad \at
		Department of Mathematics, School of Natural Sciences,\\
		National University of Sciences and Technology,
		Islamabad Pakistan \\
		Tel.: +92-333-6330824\\
		\email{hafizmuhammadfahad13@gmail.com}           
		\and
		Mujeeb ur Rehman \at
		Department of Mathematics, School of Natural Sciences,\\
		National University of Sciences and Technology,
		Islamabad Pakistan \\
		Tel.: +92-51-90855588\\
		\email{mujeeburrrehman345@yahoo.com} \\
	}
	
	\date{Received: date / Accepted: date}

	\maketitle
\begin{abstract}
In this work we focus on substantial fractional integral and differential operators which play an important role in modeling anomalous diffusion.
We  introduce a new generalized  substantial  fractional integral. Generalizations of fractional
substantial derivatives  are also introduced both in Riemann-Liouville and Caputo sense.
Furthermore, we analyze  fundamental properties of these operators.
Finally, we consider a class of generalized substantial fractional differential equations and discuss the existence, uniqueness and continuous dependence of solutions on initial data.
\end{abstract}

\section{Introduction}

Fractional calculus originated on September 30, 1695 when Leibniz expressed his idea of derivative in a note to De l’Hospital. De l’Hospital asked about the meaning of $ \frac{d^{n}f(x)}{dx^{n}} $ when $ n=\frac{1}{2} $. But by now the field of fractional calculus has been revolutionized. Nowadays this field has become very popular amongst the scientists and a great number of different forms of fractional operators have been introduced by notable researchers \cite{Ross and Samko,Podlubny,Hilfer,Kilbas2,Kilbas,R. Khalil,Abdeljawad}.

When it comes to practical applications, the substantial derivatives, introduced by  R. Friedrich et al. \cite{R. Friedrich}, have a wide range of utilization. For example, R. Friedrich, F. Jenko,  A. Baule, and S. Eule found that a fractional substantial derivative which represents important non-local couplings in space and time, is involved in generalized Fokker-Planck collision operator. By taking a modified shifted substantial Grunwald formula, Zhaopeng Hao, Wanrong Cao and Guang Lin \cite{Z. Hao} found a second-order approximation of fractional substantial derivative. Minghua Chen and Weihua Deng \cite{Chen} presented the numerical discretizations and some properties of the fractional substantial operators. Shai Carmi, Lior Turgeman and Eli Barkai \cite{Lior} used the CTRW model \cite{H. Scher} and derived forward as well as backward fractional Feynman--Kac equation by replacing the ordinary temporal derivative with substantial derivative.

The selection of a suitable fractional operator depends on the physical system under consideration. As a result, we observe numerous definitions of different fractional operators in literature. So, it is logical to establish and study the generalized fractional operators, for which the existing ones are particular cases. Katugampola recently introduced a fractional operator which generalizes the Hadamard and Riemann-Liouville fractional operators \cite{Katugampola}. In \cite{Katugampola16} he presented a more generalized fractional integral operator such that the famous fractional operators, Erdelyi-Kober, Liouville, Katugampola, Hadamard, Riemann-Liouville, and Weyl become special cases of it.  Agrawal \cite{Agrawal} presented some new operators which unified Riesz–Riemann–Liouville, the left and the right fractional Riemann–Liouville, Riesz–Caputo and Caputo derivatives and the fractional Riemann–Liouville integrals. These operators further investigated by Lupa and Klimek et al.\cite{Lupa and Klimek} and Odzijewicz et al. \cite{Odzijewicz2b}. 

In this paper, we introduce the generalized substantial fractional operators both in Riemann-Liouville and Caputo sense, and obtain the relations between the generalized substantial fractional integral and some other famous fractional integrals, namely, Riemann-Liouville type Katugampola, standard Riemann-Liouville, standard substantial and Hadamard fractional integrals. We establish the relations between the generalized substantial and Riemann-Liouville type Katugampola fractional operators. Proofs of the composition rules for the newly defined generalized operators are also the part of this work. Finally, we prove the well-posedness results for a class of generalized substantial fractional differential equations.


The paper is organized as follows. In Section 2, we state definitions and some important properties of substantial and Katugampola fractional operators. In Section 3, the generalized substantial fractional operators are introduced and fundamental properties of these operators are analyzed. Section 4 and 5 are devoted to well-posedness results for a class of generalized substantial fractional differential equations.

\section{Preliminaries}
Prior to introducing fractional differential operators, we first give some notations for sake of convenience in further developments.

The notation ${}_\sigma D^m:=\left(\frac{d}{dt}+\sigma\right)^m$ where $\left(\frac{d}{dt}+\sigma\right)^m=\left(D+\sigma\right)
\left(D+\sigma\right)\left(D+\sigma\right)\cdots
\left(D+\sigma\right)$ appears frequently in literature \cite{Chen}. Along with the operator ${}_\sigma D$, in  sequel we shall use the operator  ${}_\sigma D^{m,\rho }:=(\frac{t^{1-\rho}}{\rho}\frac{d}{dt}+\sigma)^m  $, where $\sigma\in\mathbb{R}$ and $\rho\neq0$. Also the  generalized
differential operator  defined as $\frac{t^{1-\rho}}{\rho}\frac{d^m}{dt^m}$, will be denoted by $D^{m,\rho}$.
We define  function spaces $\Omega_{\sigma,\rho}^m [a,b]:=\left \{ \psi:e^{\sigma t^{\rho}}t^{1-\rho}\psi(t) \in AC^m[a,b] \right\}$ and
$\Lambda_{\sigma,\rho}^{p}[a,b] :=\left \{ \psi:e^{\sigma t^{\rho}}t^{1-\rho}\psi \in L_{p}[a,b] \right \} $ where $AC[a,b]$ is the space of absolutely continuous functions and $L_{p}[a,b]~(1\leq p<\infty)$ denotes the space of measurable functions on $[a,b]$. For simplicity $\Omega_{\sigma,1}$ and $\Omega_{0,\rho}$  will be denoted by $\Omega_{\sigma}$ and $\Omega_{\rho}$ respectively.
\subsection{Substantial fractional operators}
\begin{definition}\label{def1}\cite{R. Friedrich,Chen}
	Let $\alpha$ and $ \sigma $ be real numbers such that $ \alpha > 0 $ and   $ \psi \in \Lambda_{\sigma}^{1}[a,b]$ then substantial  fractional integral operator is defined as
	$\tilde{I}^{\alpha}_{a}\psi(t)=\frac{1}{\Gamma(\alpha)}\int_{a}^{t} (t - s)^{\alpha-1} e^{-\sigma (t-s)} \psi(s) ds$.
	Furthermore for $\psi\in \Omega_{\sigma}^{m}[a,b]$, $m-1<\alpha\leq m$, the Riemann-Liouville type substantial fractional derivative is defined as $\tilde{D}^{\alpha}_{a}\psi(t)=\tilde{D}^{m}_{a}{}^{}_{}\tilde{I}^{m-\alpha}_{a}\psi(t).$ The Caputo type substantial fractional derivative is defined as ${}^c\tilde{D}^{\alpha}_{a}\psi(t)=\tilde{I}^{m-\alpha}_{a}\tilde{D}^{m}_{a}\psi(t).$
\end{definition}

\subsection{Katugampola fractional operators}
For $\rho\neq0$, let $I^{1,\rho}_{a}\psi(t)=\int_{a}^{t}\psi(s)d(s^\rho)$, where $d(s^\rho)=\rho s^{\rho-1}ds$. Then $m$th iterate of the integral operator $I^{1,\rho}_{a}$ is given by
\begin{equation}\label{Int1}
\begin{split}
I^{m,\rho}_{a}\psi(t)&=\int_{a}^{t}d(t_1^\rho)\int_{a}^{t_1}d(t_2^\rho)\int_{a}^{t_2}d(t_3^\rho)
\cdots \int_{a}^{t_{n-1}}\psi(t_{n-1})d(t_{n-1}^\rho)\\
&=\frac{\rho}{\Gamma(m)}\int_{a}^{t}(t^\rho-s^\rho)^{m-1}\psi(s)s^{\rho-1}ds.
\end{split}
\end{equation}
Replacing the $m$ by real $\alpha>0$ in \eqref{Int1}, the Katugampola fractional integral is defined as
\begin{definition} \label{KID}\cite{{Katugampola}} For $\rho\neq0$, $\alpha>0$ and $\psi\in\Lambda_{\rho}^{1}[a,b]$, the Katugampola fractional integral is given by
	\begin{equation*}
	I^{\alpha,\rho}_{a}\psi(t)=\frac{\rho}{\Gamma(\alpha)}\int_{a}^{t}(t^\rho-s^\rho)^{\alpha-1}\psi(s)s^{\rho-1}ds.
	\end{equation*}
	Furthermore, for $m-1<\alpha\leq m$ and  $\phi\in \Omega_{\rho}^{m}[a,b]$ the Riemann-Liouville type  Katugampola fractional derivative is defined as
	$D^{\alpha,\rho}_{a}\psi(t)=D^{m,\rho}I^{m-\alpha,\rho}_{a}\psi(t)$ and the Caputo-type  Katugampola is defined as
	${}^cD_{a}^{\alpha}\psi(t)=I^{m-\alpha,\rho}_{a}D^{m,\rho}\psi(t)$.
\end{definition}
\begin{remark} 
	We introduced a slight modification in the definition of Katugampola fractional integral operator. The factor $\rho^{1-\alpha}$ in original definition is now replaced with $\rho$. This avoids repeated appearance of some factors of $\rho$ in calculations \cite[p.~103]{Kilbas}.
\end{remark}
It is to be noted $ D^{1,\rho} I^{1,\rho}_{a}\psi(t)=\psi(t)$. A repeated application of this identity leads us to the identity
$ D^{m,\rho} I^{m,\rho}_{a}\psi(t)=\psi(t)$. Furthermore
$I^{1,\rho} _{a}D^{1,\rho}  \psi(t)=\int_a^tD^\rho \psi(s)d(s^\rho)=\int_a^td(\psi(s))=\psi(t)-\psi(a).$ Similarly
$I^{2,\rho} _{a}D^{2,\rho}  \psi(t)=\psi(t)-\psi(a)-(t^\rho-a^\rho)D^{1,\rho}\psi(a).$
In general, a repeated application of preceding steps leads to Taylor type expansion of $\psi$ as
\begin{equation}\label{TS1}
\psi(t)=\sum_{k=0}^{m-1}\frac{D^{k,\rho}\psi(a)}{k!}(t^\rho-a^{\rho})^k+  I_{a}^{m,\rho}  D^{m,\rho}  \psi(t).
\end{equation}
The Katugampola fractional differential and integral operators
satisfy following properties \cite{Malinowska,Oliveira}:
\begin{itemize}
	\item[(P1)] For $\psi\in\Lambda_{\rho}^{1}[a,b]$,  ${I}_a^{\alpha,\rho}  I_a^{\beta,\rho}\psi(t)
	={I}_a^{\beta,\rho} I_a^{\alpha,\rho}\psi(t)= I_a^{\alpha+\beta,\rho}\psi(t)$.
	\item[(P2)] For $\beta\geq \alpha$, and $\psi\in\Lambda_{\rho}^{1}[a,b]$,  $ D_a^{\alpha,\rho}  I_a^{\beta,\rho}\psi(t)=
	I_a^{\beta-\alpha,\rho}\psi(t)$ and  ${}^cD_a^{\alpha,\rho} I_a^{\beta,\rho}\psi(t)=I_a^{\beta-\alpha,\rho}\psi(t)$.
	\item[(P3)] For $\alpha>1$ and $m-1< \alpha\leq m$ and $\psi\in\Omega_{\rho}^{m}[a,b]$, we have 
	$$
	I_{a}^{\alpha,\rho}  D_a^{\alpha,\rho}  \psi(t)=\psi(t)- \sum_{k=1}^{m-1}  \frac{ \lim_{s\to a^+} D_{a}^{\alpha-k,\rho}\psi(s)}{\Gamma(\alpha-k+1)}(t^\rho-a^{\rho})^{\alpha-1}.$$
	Specifically, for $ 0 < \alpha < 1 $,
	$
	I_{a}^{\alpha,\rho}  D_a^{\alpha,\rho}  \psi(t)=\psi(t)-  \frac{I^{1-\alpha,\rho}_{a}\psi(a)}{\Gamma(\alpha)}(t^\rho-a^{\rho})^{\alpha-1}$.
	\item[(P4)] For $\psi\in\Omega_{\rho}^{m}[a,b]$ and $\beta\geq \alpha$,
	$$
	I_{a}^{\alpha,\rho}  {}^c D_a^{\alpha,\rho}  \psi(t)=\psi(t)- \sum_{k=0}^{m-1} \frac{D^{k,\rho}\psi(a)}{k!}(t^\rho-a^{\rho})^k.$$
\end{itemize}
\begin{lemma}\label{Lem1}  Assume that $\psi\in\Omega_{\rho}^{m}[a,b]$. Then
	$
	{}_\sigma D^{m,\rho} (e^{-\sigma t^\rho}\psi(t))=e^{-\sigma t^\rho}D^{m,\rho}\psi(t)
	$ and $  e^{\sigma t^\rho} {}_\sigma D^{m,\rho} (\psi(t))=D^{m,\rho}(e^{\sigma t^\rho}\psi(t)) $.
\end{lemma}
\begin{proof}
	We prove this Lemma by induction. For $m=1$, we have
	$$ {}_\sigma D^{1,\rho} (e^{-\sigma t^\rho}\psi(t))=\frac{t^{1-\rho}}{\rho}\frac{d}{dt}\left(e^{-\sigma t^\rho}\psi(t)\right)+\sigma e^{-\sigma t^\rho}\psi(t)=e^{-\sigma t^\rho}D^{1,\rho}\psi(t).$$
	Assume the conclusion follows for $m-1$. Then,
	$$
	{}_\sigma D^{m,\rho} (e^{-\sigma t^\rho}\psi(t))={}_{\sigma} D^{1,\rho}{}_{\sigma}D^{m-1,\rho}( e^{-\sigma t^\rho}\psi(t))                                  ={}_{\sigma} D^{1,\rho}\left(e^{-\sigma t^\rho}D^{m-1,\rho}\psi(t)\right)=e^{-\sigma t^\rho}D^{m,\rho}\psi(t).$$ Second identity can be obtained in a similar way.
\end{proof}
%
\section{Generalized substantial fractional integral and derivatives}
Motivated by  definitions of substantial fractional operators, here we introduce  new definitions for substantial fractional operators by generalizing   Katugampola fractional operators. We also establish relation between generalized substantial fractional operators and the Katugampola fractional operators. \\
For $\rho\neq0$ and $\sigma\in \mathbb{R}$, define ${}_\sigma I_{a}^{1,\rho}\psi(t)=\int_{a}^{t}\psi(s)e^{-\sigma(t^\rho-s^\rho)}d(s^\rho)$. Then generalized substantial integral of order $m$ is given by $m$th
iterate of the integral ${}_\sigma I_{a}^{1,\rho}$ as
\begin{equation}\label{IntSBS}
\begin{split}
{}_\sigma I_{a}^{m,\rho}\psi(t)&=\int_{a}^{t}e^{-\sigma(t^\rho-t_1^\rho)}d(t_1^\rho)
\int_{a}^{t_1}e^{-\sigma(t^\rho-t_2^\rho)}d(t_2^\rho)
\cdots \int_{a}^{t_{n-1}}e^{-\sigma(t^\rho-t_{n-1}^\rho)}\psi(t_{n-1})d(t_{n-1}^\rho)\\
&=\frac{\rho}{\Gamma(m)}\int_{a}^{t}(t^\rho-s^\rho)^{m-1}e^{-\sigma(t^\rho-s^\rho)}\psi(s)s^{\rho-1}ds.
\end{split}
\end{equation}

We observe that ${}_\sigma D^{1,\rho}{}_\sigma I^{1,\rho}_{a}\psi(t)=\psi(t)$. A repeated application of this identity leads us to the identity
${}_\sigma D^{m,\rho}{}_\sigma I^{m,\rho}_{a}\psi(t)=\psi(t)$.
Thus for $m>n$, we have  ${}_\sigma D^{m-n,\rho}{}_\sigma I^{m-n,\rho}_{a}\psi(t)=\psi(t)$. Application of the operator
${}_\sigma D^{n,\rho}$ to both sides of this identity leads to the identity ${}_\sigma D^{n,\rho}\psi(t)={}_\sigma D^{m,\rho}{}_\sigma I^{m-n,\rho}_{a}\psi(t).$ This relation will lead us to the definition of generalized fractional derivative.
Furthermore
${}_\sigma I^{1,\rho}_{a}{}_\sigma D^{1,\rho} \psi(t)
=\int_a^t (\frac{s^{1-\rho}}{\rho}\frac{d}{ds}+\sigma)\psi(s)d(s^\rho)
=\psi(t)-\psi(a)e^{-\sigma(t^\rho-a^\rho)}.$ In general a repeated application of this process leads us to   generalized Taylor expansion involving generalized operators
\begin{equation}\label{TS2}
{}_\sigma I^{m,\rho}_{a}{}_\sigma D^{m,\rho}\psi(t)=\psi(t)-  e^{-\sigma (t^\rho - a^\rho)} \sum_{k=1}^{m} 
\frac{(t^\rho-a^{\rho})^{m-k} }{\Gamma(m-k+1)}\underset{s\to a^+}{\lim} {}_\sigma D^{m-k,\rho} \psi(s)\nonumber
\end{equation}
provided $\psi\in \Omega_{\rho}^{m}[a,b]$.
\begin{definition}\label{GSID}
	For real numbers $ \sigma $, $ \rho\neq 0 $, $\alpha>0$ and   $ \psi \in \Lambda^1_{\sigma,\rho}[a,b]$,  we define generalized substantial integral as
	$${}_\sigma I_a^{\alpha,\rho }\psi(t)=\frac{\rho}{\Gamma(\alpha)}\int_a^t\frac{s^{\rho-1}
		e^{-\sigma(t^\rho-s^\rho)}}{(t^\rho-s^\rho)^{1-\alpha}}\psi(s) ds. $$
	Furthermore, the Riemann-Liouville type generalized substantial fractional derivative is defined as
	$
	{}_\sigma D_a^{\alpha,\rho }\psi(t)={}_\sigma D^{m,\rho } {}_\sigma I_a^{m-\alpha, \rho}\psi(t) $
	where $m-1<\alpha\leq m$.
\end{definition}

%
It is to be noted that for $\sigma\to0$, ${}_\sigma I_a^{\alpha,\rho }\psi(t)\to I_a^{\alpha,\rho }\psi(t)$, which is Katugampola fractional integral. Furthermore, for $\sigma=0$ and $\rho\to 1$, the generalized  substantial integral approaches to standard Riemann-Liouville integral and the lower limit $a\to-\infty$ leads to the Weyl fractional integral. For $\sigma\neq 0$ and $\rho=1$ the generalized  substantial integral becomes the standard substantial integral. Finally for $\sigma=0$ and $\rho\to 0$, we get the Hadamard fractional integral.

\begin{definition}\label{def5}
	For  $m-1<\alpha\leq m$,  $a<b<\infty$ and $ \psi\in \Omega_{\sigma,\rho}^m[a,b]$.
	Then generalized Caputo type substantial  derivative is defined  as
	\begin{align*}
	{}_\sigma^c{D}_a^{\alpha,\rho}\psi(t)={}_\sigma{D}_a^{\alpha,\rho}\Big(\psi(t)-\sum_{k=0}^{m-1}\frac{{}_\sigma{D}^{k,\rho}\psi(a)}{k!}(t^\rho-a^\rho)^k e^{-\sigma(t^\rho-a^\rho)}\Big).
	\end{align*}
\end{definition}\label{GSCaputo}

\begin{theorem}\label{theorem4} Assume $ \alpha>0 $, $ \sigma>0 $, $\rho > 0$ and $ \{\psi_k\}_\text{k=$1$}^\infty $ is a uniformly convergent sequence of continuous functions on $[0,b]$. Then
	$
	(  {}_\sigma I_0^{\alpha,\rho}\lim_{k\to\infty} \psi_k)(t)=(\lim_{k\to\infty}\   {}_\sigma I_0^{\alpha,\rho} \psi_k)(t).
	$
\end{theorem}
\begin{proof}
	We denote the limit of sequence $ \{\psi_k\}_\text{k=$1$}^\infty $  by $ \psi $. It is well-known that $ \psi $ is continuous. We then have following estimates
	\begin{equation}
	\begin{split}
	|{}_\sigma I_0^{\alpha,\rho}\psi_k(t)- {}_\sigma I_0^{\alpha,\rho}\psi(t)|
	&\leq\frac{\rho}{\Gamma(\alpha)}\int_{0}^{t}s^{\rho-1} (t^\rho - s^\rho)^{\alpha-1} | ( e^{-\sigma (t^\rho-s^\rho)}  ) (\psi_k(s)-\psi(s))|  ds
	\\&\leq  \frac{\rho }{\Gamma(\alpha)}\lVert \psi_k - \psi \rVert_{\infty}  \int_{0}^{t} s^{\rho-1} (t^\rho - s^\rho)^{\alpha-1}ds
	\\&=     \frac{{b^\rho}}{\Gamma(\alpha+1)}\lVert \psi_k - \psi \rVert_{\infty}.
	\nonumber
	\end{split}
	\end{equation}
	The conclusion follows, since $\lVert \psi_k - \psi \rVert_{\infty}\to 0$  as $ {k\to\infty} $ uniformly  on $[0,b] $.
\end{proof}
In the forthcoming results, we shall demonstrate the relationship between Riemann-Liouville type Katugampola fractional operators and the generalized substantial fractional operators.
\begin{lemma}\label{SFI} Assumptions $\psi\in\Lambda_{\sigma,\rho}^1[a,b]$. Then
	${}_{\sigma}I_a^{\alpha,\rho}\psi(t)=e^{-\sigma t^\rho}I_a^{\alpha,\rho}(e^{\sigma t^\rho}\psi(t)).$
\end{lemma}
\begin{theorem}\label{Th5}
	Assumptions $\psi\in \Omega_{\sigma,\rho}^m[a,b]$. Then
	$
	_{\sigma}D_{a}^{\alpha,\rho}\psi(t)=e^{-\sigma t^\rho}D_{a}^{\alpha,\rho}(e^{\sigma t^\rho}\psi(t)).
	$
\end{theorem}
\begin{proof}
	By definition \eqref{GSID}   of substantial fractional differential operator,  Lemma \ref{Lem1}, Lemma \ref{SFI} and definition \ref{KID} we have
	\begin{align*}
	_{\sigma}D_{a}^{\alpha,\rho}\psi(t)&={}_{\sigma}D^{m,\rho} {}_{\sigma}I_a^{m-\alpha,\rho}\psi(t)
	={}_{\sigma}D^{m,\rho}\left(e^{-\sigma t^\rho}I_a^{m-\alpha,\rho}(e^{\sigma t^\rho}\psi(t))\right)\\
	&=e^{-\sigma t^\rho}D^{m,\rho}I_a^{m-\alpha,\rho}(e^{\sigma t^\rho}\psi(t))
	=e^{-\sigma t^\rho}D_a^{\alpha,\rho}(e^{\sigma t^\rho}\psi(t)).
	\end{align*}
\end{proof}
Now we introduce composition properties of the generalized substantial operators. First we show that generalized  integral satisfies the semi-group property.
\begin{theorem}\label{theorem3} Let $  \alpha , \beta > 0 $ and  $\psi\in\Lambda_{\sigma,\rho}^1[a,b]$. Then
	$
	{}_\sigma I_a^{\alpha,\rho}{}_\sigma I_a^{\beta,\rho}\psi(t)={}_\sigma I_a^{\alpha+\beta,\rho}\psi(t). \nonumber
	$
\end{theorem}
\begin{proof} Using $(P1)$ and  Lemma \ref{SFI} repeatedly we have
	\begin{equation}
	\begin{split}
	{}_\sigma I_a^{\alpha,\rho }({}_\sigma I_a^{\beta,\rho }\psi(t))&=  {}_\sigma I_a^{\alpha,\rho } ( e^{-\sigma t^\rho}      I_a^{\beta,\rho } (e^{\sigma t^\rho}\psi(t) ) )
	=e^{-\sigma t^\rho}  I_a^{\alpha+\beta,\rho } (e^{\sigma t^\rho}\psi(t))={}_\sigma I_a^{\alpha+\beta,\rho}\psi(t).  \nonumber
	\end{split}
	\end{equation}
\end{proof}
\begin{theorem}\label{theorem2} Let $m-1<\alpha\leq m$, $\beta\geq \alpha$ and  $\psi\in\Lambda_{\sigma,\rho}^1[a,b]$. Then
	$
	{}_\sigma D_a^{\alpha,\rho } {}_\sigma I_a^{\beta,\rho}\psi(t)={}_\sigma I_a^{\beta-\alpha,\rho}\psi(t).
	$
\end{theorem}
The proof of Theorem \ref{theorem2} is same as the proof of the Theorem \ref{theorem3}. Therefore we omit it.
\begin{theorem}\label{theorem2B} Assume $ \alpha>0 $, $m-1<\alpha\leq m$ and  $ {}_\sigma I_a^{m-\alpha,\rho}\psi\in \Omega_{\sigma,\rho}^m[a,b]$.
	Then $$
	{}_\sigma I_a^{\alpha,\rho}{}_\sigma D_a^{\alpha,\rho } \psi(t)=\psi(t)-  e^{-\sigma (t^\rho - a^\rho)} \sum_{k=1}^{m} 
	\frac{(t^\rho-a^{\rho})^{\alpha-k} }{\Gamma(\alpha-k+1)}\underset{s\to a^+}{\lim} {}_\sigma D_{a}^{\alpha-k,\rho} \psi(s)\nonumber
	.$$
	
	Specifically, for $ 0<\alpha<1 $ we have
	$$
	{}_\sigma I_a^{\alpha,\rho}{}_\sigma D_a^{\alpha,\rho } \psi(t)=\psi(t)-  e^{-\sigma (t^\rho - a^\rho)} 
	\frac{(t^\rho-a^{\rho})^{\alpha-1} }{\Gamma(\alpha)}\underset{s\to a^+}{\lim} {}_\sigma I_{a}^{1-\alpha,\rho} \psi(s). \nonumber$$
\end{theorem}
\begin{proof}
	\allowdisplaybreaks
	Using Leibniz rule, following relation can be established.
	\begin{align}
	\Big(\frac{t^{1-\rho}}{\rho}\frac{d}{dt}+\sigma\Big) \int_a^t\frac{s^{\rho-1}
		e^{-\sigma(t^\rho-s^\rho)}}{(t^\rho-s^\rho)^{-\alpha}}{}_\sigma D_a^{\alpha,\rho }\psi(s) ds &= \alpha \int_a^t\frac{s^{\rho-1}
		e^{-\sigma(t^\rho-s^\rho)}}{(t^\rho-s^\rho)^{1-\alpha}}{}_\sigma D_a^{\alpha,\rho }\psi(s) ds. \label{equation1}
	\end{align}
	By definition of $ {}_\sigma I_a^{\alpha,\rho}$, we have
	\begin{align}
	{}_\sigma I_a^{\alpha,\rho}{}_\sigma D_a^{\alpha,\rho } \psi(t)&=\frac{\rho}{\Gamma(\alpha)}\int_a^t\frac{s^{\rho-1}
		e^{-\sigma(t^\rho-s^\rho)}}{(t^\rho-s^\rho)^{1-\alpha}}{}_\sigma D_a^{\alpha,\rho }\psi(s) ds. \label{equation2}
	\end{align}
	From Eq. \eqref{equation1} and \eqref{equation2}, we get
	\begin{align}
	{}_\sigma I_a^{\alpha,\rho}{}_\sigma D_a^{\alpha,\rho } \psi(t)&=\frac{\rho}{\Gamma(\alpha+1)} \Big(\frac{t^{1-\rho}}{\rho}\frac{d}{dt}+\sigma\Big) \int_a^t\frac{s^{\rho-1}
		e^{-\sigma(t^\rho-s^\rho)}}{(t^\rho-s^\rho)^{-\alpha}}{}_\sigma D_a^{\alpha,\rho }\psi(s) ds. \label{equation3}
	\end{align}
	From Definition \ref{GSID} and Eq. \eqref{equation3}, we find
	\begin{align*}
	{}_\sigma I_a^{\alpha,\rho}{}_\sigma D_a^{\alpha,\rho } \psi(t)&=\frac{\rho}{\Gamma(\alpha+1)} \Big(\frac{t^{1-\rho}}{\rho}\frac{d}{dt}+\sigma\Big) \int_a^t\frac{s^{\rho-1}
		e^{-\sigma(t^\rho-s^\rho)}}{(t^\rho-s^\rho)^{-\alpha}} \Big(\frac{s^{1-\rho}}{\rho}\frac{d}{ds}+\sigma\Big) {}_\sigma D^{m-1,\rho } {}_\sigma I_a^{m-\alpha,\rho} \psi(s) ds. 
	\end{align*}
	Applying integration by parts and product rule for classical derivatives, we have
	\begin{align*}
	{}_\sigma I_a^{\alpha,\rho}{}_\sigma D_a^{\alpha,\rho } \psi(t)&= - \frac{	e^{-\sigma(t^\rho-a^\rho)}}{\Gamma(\alpha) (t^\rho-s^\rho)^{1-\alpha}}\underset{s\to a^+}{\lim} {}_\sigma D^{m-1,\rho } {}_\sigma I_a^{m-\alpha,\rho} \psi(s) \\
	&\quad + \frac{\rho}{\Gamma(\alpha)} \Big(\frac{t^{1-\rho}}{\rho}\frac{d}{dt}+\sigma\Big) \int_a^t\frac{s^{\rho-1}
		e^{-\sigma(t^\rho-s^\rho)}}{(t^\rho-s^\rho)^{1-\alpha}} \Big(\frac{s^{1-\rho}}{\rho}\frac{d}{ds}+\sigma\Big)  {}_\sigma D^{m-2,\rho } {}_\sigma I_a^{m-\alpha,\rho} \psi(s) ds.  
	\end{align*}
	Continuing in this manner, we get
	\begin{align*}
	{}_\sigma I_a^{\alpha,\rho}{}_\sigma D_a^{\alpha,\rho } \psi(t)&=-e^{-\sigma (t^\rho - a^\rho)} \sum_{k=1}^{m}  \frac{(t^\rho-a^{\rho})^{\alpha-k} }{\Gamma(\alpha-k+1)}\underset{s\to a^+}{\lim} {}_\sigma D_{a}^{\alpha-k,\rho} \psi(s) \\
	&\quad  + \frac{\rho}{\Gamma(\alpha-m+1)} \Big(\frac{t^{1-\rho}}{\rho}\frac{d}{dt}+\sigma\Big) \int_a^t\frac{s^{\rho-1}
		e^{-\sigma(t^\rho-s^\rho)}}{(t^\rho-s^\rho)^{m-\alpha}}   {}_\sigma I_a^{m-\alpha,\rho} \psi(s) ds, 
	\end{align*}
	where 
	\begin{align*}
	\frac{\rho}{\Gamma(\alpha-m+1)} \Big(\frac{t^{1-\rho}}{\rho}\frac{d}{dt}+\sigma\Big) \int_a^t\frac{s^{\rho-1}
		e^{-\sigma(t^\rho-s^\rho)}}{(t^\rho-s^\rho)^{m-\alpha}}   {}_\sigma I_a^{m-\alpha,\rho} \psi(s) ds&=\psi(t). 
	\end{align*}
	Finally, we get the desired result
	$$
	{}_\sigma I_a^{\alpha,\rho}{}_\sigma D_a^{\alpha,\rho } \psi(t)=\psi(t)-  e^{-\sigma (t^\rho - a^\rho)} \sum_{k=1}^{m} 
	\frac{(t^\rho-a^{\rho})^{\alpha-k} }{\Gamma(\alpha-k+1)}\underset{s\to a^+}{\lim} {}_\sigma D_{a}^{\alpha-k,\rho} \psi(s)\nonumber
	.$$
\end{proof}

\begin{theorem}\label{Th6} Assume $ \psi\in \Omega_{\sigma,\rho}^m[a,b]$.
	Then generalized Caputo type substantial  derivative can be written as\\
	$$
	{}_\sigma^c{D}_a^{\alpha,\rho}\psi(t)={}_\sigma {I}_a^{m-\alpha,\rho}{}_\sigma D^{m,\rho}\psi(t)
	=\frac{\rho}{\Gamma(m-\alpha)}
	\int_a^t\frac{e^{-\sigma(t^\rho-s^\rho)}}{(t^\rho-s^\rho)^{1+\alpha-m}}{}_\sigma D^{m,\rho}(\psi(s)) ds.
	$$
\end{theorem}
\begin{proof}
	By definition \ref{def5} and Equation \eqref{TS2} we have
	${}_\sigma^c{D}_a^{\alpha,\rho}\psi(t)={}_\sigma{D}_a^{\alpha,\rho}{}_\sigma I^{m,\rho}{}_\sigma D^{m,\rho}\psi(t)$. An application of definition \ref{GSCaputo}, and properties $(P1)$ and $(P2)$   leads us to 
	\begin{align*}
	{}_\sigma^c{D}_a^{\alpha,\rho}\psi(t)
	&={}_\sigma{D}^{m,\rho}{}_\sigma {}I_a^{m-\alpha,\rho}{}_\sigma I_a^{m,\rho}{}_\sigma D^{m,\rho}\psi(t)
	={}_\sigma{D}^{m,\rho}{} {}_\sigma{I}_a^{m,\rho} {}_\sigma I_a^{m-\alpha,\rho}{}_\sigma D^{m,\rho}\psi(t)\\
	&={}_\sigma I_a^{m-\alpha,\rho}{}_\sigma D^{m,\rho}\psi(t).
	\end{align*}
\end{proof}
\begin{lemma} \label{Lm26} For  $ \psi\in \Omega_{\rho}^m[a,b]$, the  operator ${}_\sigma^c{D}_a^{\alpha,\rho}$ satisfies the relation
	${}_\sigma^c{D}_a^{\alpha,\rho}\psi(t)=e^{-\sigma t^\rho}{}^cD_{a}^{\alpha,\rho}(e^{\sigma t^\rho}\psi(t)).$
\end{lemma}
\begin{proof} By using Lemma \ref{Lem1}, Lemma \ref{SFI} and Theorem \ref{Th6} we have
	\begin{align*}
	{}_\sigma^c{D}_a^{\alpha,\rho}\psi(t)&={}_\sigma I_a^{m-\alpha,\rho}{}_\sigma D^{m,\rho}\psi(t)
	=e^{-\sigma t^\rho}{I}_a^{m-\alpha,\rho}(e^{\sigma t^\rho}{}_\sigma D^{m,\rho}\psi(t))\\
	&=e^{-\sigma t^\rho}{I}_a^{m-\alpha,\rho}{}D^{m,\rho}(e^{\sigma t^\rho}\psi(t))
	=e^{-\sigma t^\rho}{} ^c{D}_a^{\alpha,\rho}(e^{\sigma t^\rho}\psi(t)).
	\end{align*}
\end{proof}

\begin{theorem}\label{theorem5} Let $m-1<\alpha\leq m$, $\beta\geq \alpha$ and $ \psi \in \Omega_{\sigma,\rho}^m [a,b]$. Then
	$
	{}_\sigma^c{D}_a^{\alpha,\rho} {}_\sigma I_a^{\beta,\rho} \psi(t)={}_\sigma I_a^{\beta-\alpha,\rho}\psi(t).
	$
\end{theorem}
By using Lemma \ref{SFI}  and Lemma \ref{Lm26}, the result can easily be proved.

\begin{theorem} \label{caputoid}
	Assume $m-1<\alpha\leq m$ and   $ \psi\in \Omega_{\sigma,\rho}^m[a,b]$.
	Then $$  {}_\sigma{I}_a^{\alpha,\rho}{}_\sigma^c{D}_a^{\alpha,\rho}\psi(t)
	=\psi(t)-\sum_{k=0}^{m-1}\frac{{}  {}_\sigma{D}^{k,\rho}\psi(a)}{k!}e^{-\sigma(t^\rho-a^\rho)}(t^\rho-a^\rho)^k.$$
\end{theorem}
\begin{proof}
	From Lemma \ref{Lm26} and Lemma \ref{SFI} we have \\
	${}_\sigma{I}_a^{\alpha,\rho}{}_\sigma^c{D}_a^{\alpha,\rho}\psi(t)
	={}_\sigma{I}_a^{\alpha,\rho}\left(e^{-\sigma t^\rho}{}^c{D}_a^{\alpha,\rho}(e^{\sigma t^\rho}\psi(t))\right)
	=e^{-\sigma t^\rho}{I}_a^{\alpha,\rho}({}^c{D}_a^{\alpha,\rho}(e^{\sigma t^\rho}\psi(t))).$\\
	Now by property $(P4)$ we have
	\begin{align*}
	{}_\sigma{I}_a^{\alpha,\rho}{}_\sigma^c{D}_a^{\alpha,\rho}\psi(t)&=
	e^{-\sigma t^\rho}\Big(e^{\sigma t^\rho}\psi(t)-\sum_{k=0}^{m-1}\frac{ D^{k,\rho}(e^{\sigma t^\rho}\psi(t))|_{t=a}}
	{k!}(t^\rho-a^\rho)^k\Big)
	\\&=\psi(t)-\sum_{k=0}^{m-1}\frac{{}  {}_\sigma D^{k,\rho}\psi(a)}{k!}e^{-\sigma(t^\rho-a^\rho)}(t^\rho-a^\rho)^k.
	\end{align*}
\end{proof}
\begin{figure}[h!]
	\centering
	\subfloat[ $\beta=2, \sigma=1, \rho=1.5$, $0.7\leq\alpha\leq 1$.]
	{\label{Fig1a}\includegraphics[width=0.5\textwidth]{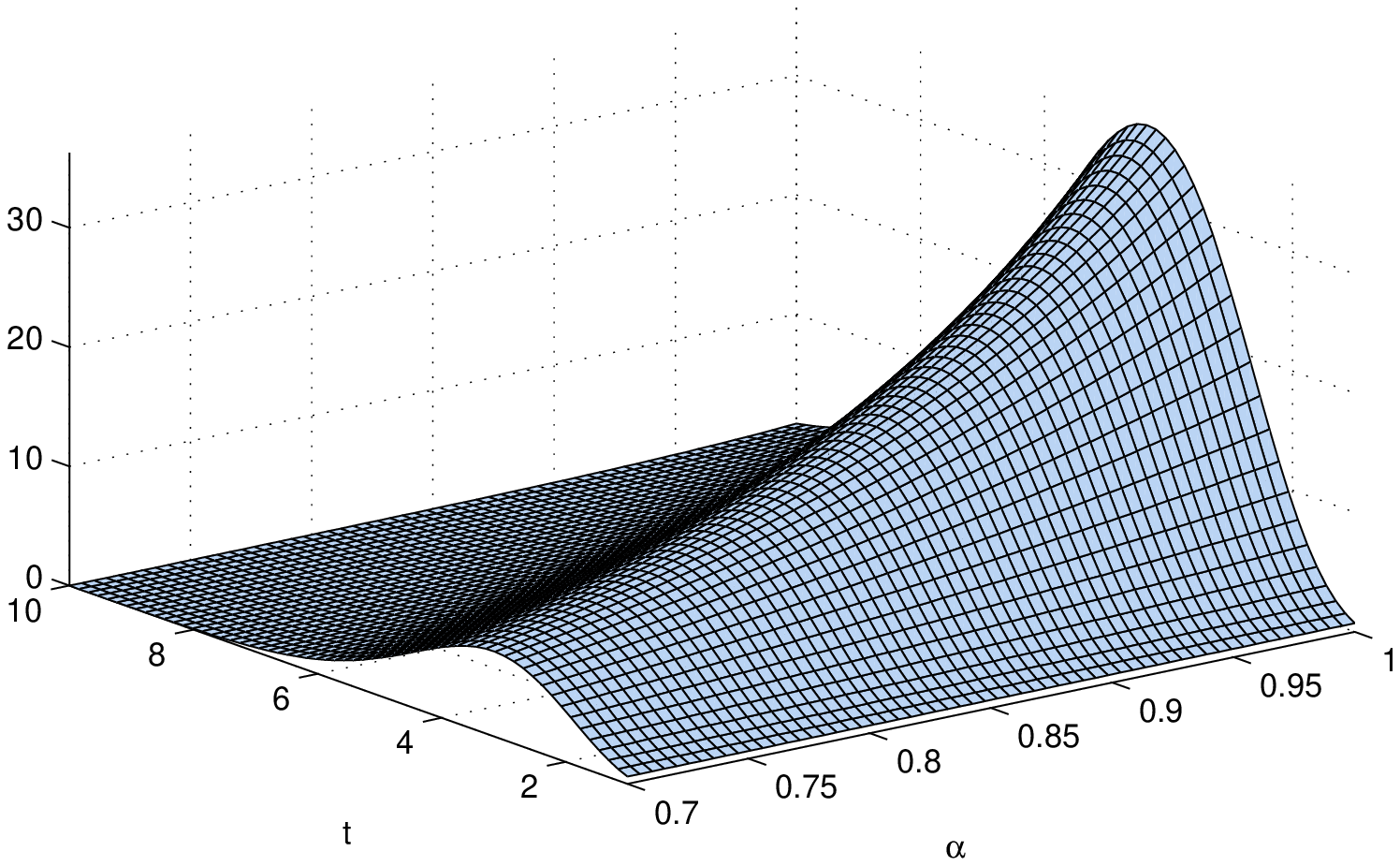}}
	\subfloat[$\alpha=0.5, \beta=2,\sigma=1, 1< \rho\leq 5$.]
	{\label{Fig1b}\includegraphics[width=0.5\textwidth]{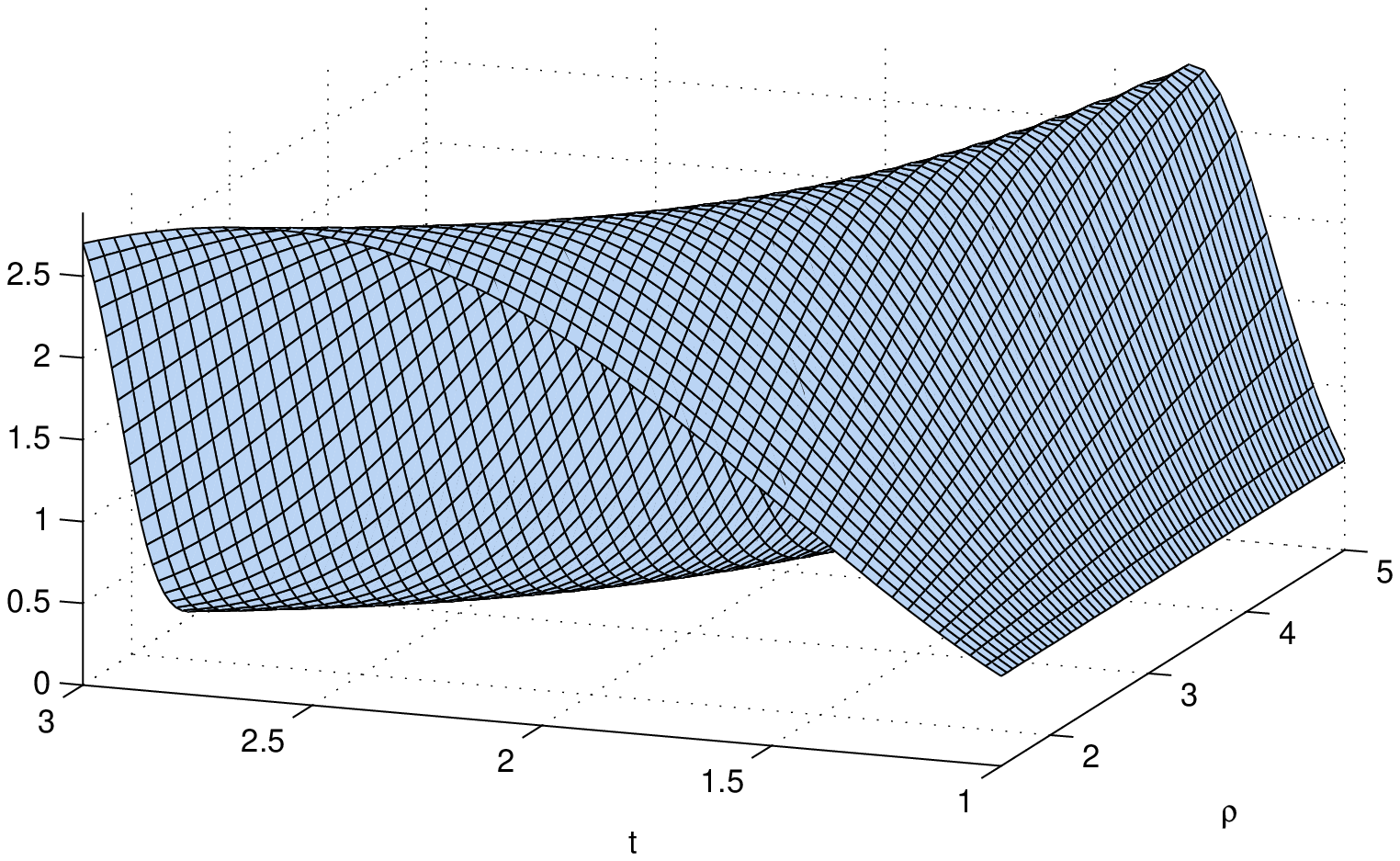}}\\
	\subfloat[$\alpha=0.5, \beta=2, \rho=2, 1.5\leq \sigma\leq 5$.]
	{\label{Fig1c}\includegraphics[width=0.5\textwidth]{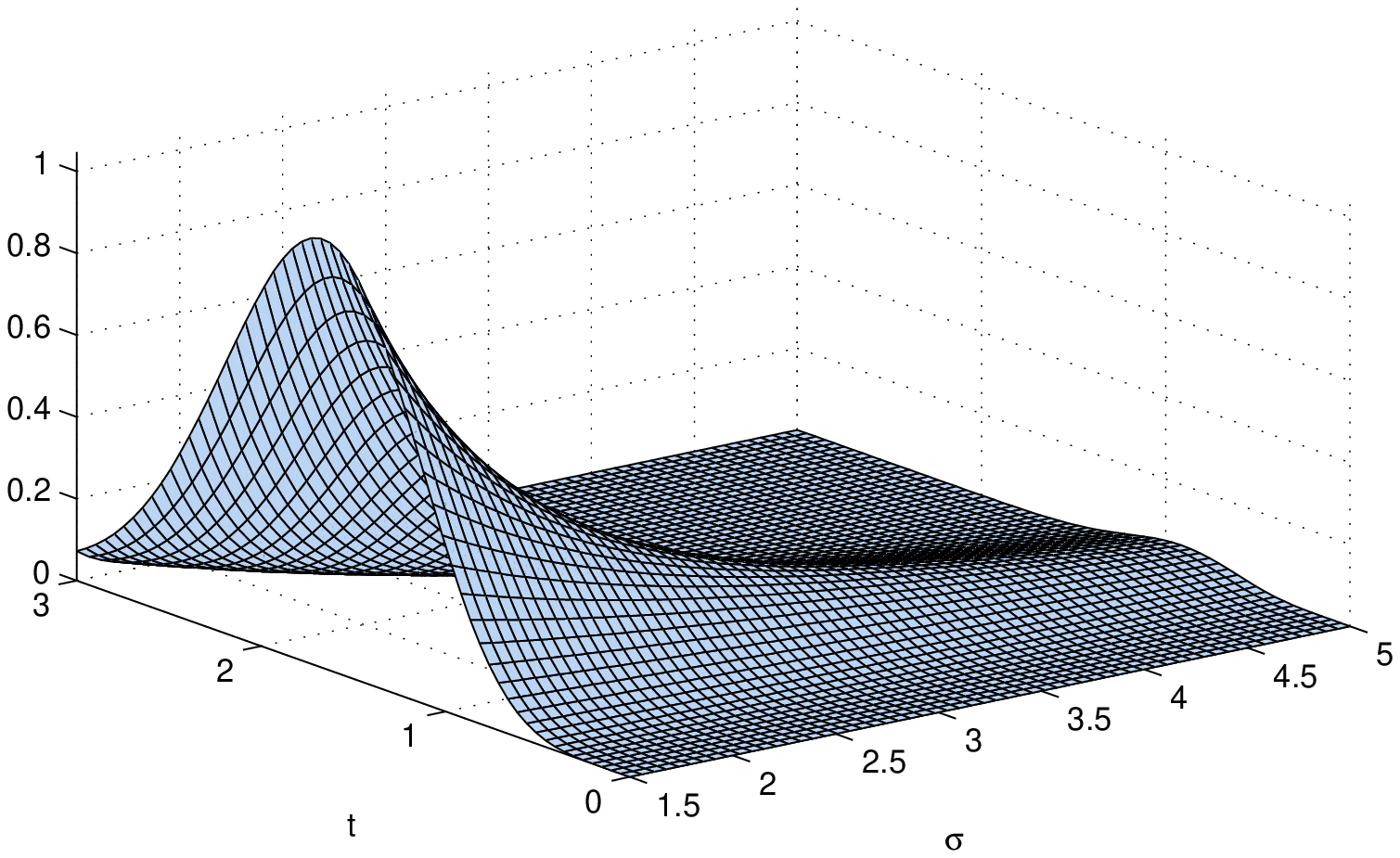}}
	\caption{Fractional integrals ${}_\sigma I_a^{\alpha,\rho}$ of  $\psi(t)=(t^\rho-a^\rho)^{\beta}e^{-\sigma t^\rho}$. }
	\label{Fig1}
\end{figure}
\begin{example}\label{exp} Consider $\psi(t)=(t^\rho-a^\rho)^{\beta}e^{-\sigma t^\rho}$. Then from Lemma \ref{SFI} we have
	${}_\sigma I_a^{\alpha,\rho}\psi(t)=e^{-\sigma t^{\rho}}I_a^{\alpha,\rho}(t^\rho-a^\rho)^{\beta}$. Now by Lemma 3 in \cite{Oliveira} we have 
	\begin{equation}\label{EX1}
	{}_\sigma I_a^{\alpha,\rho}\psi(t)=\frac{\Gamma(\beta+1)e^{-\sigma t^{\rho}}}{\Gamma(\alpha+\beta+1)}(t^\rho-a^\rho)^{\alpha+\beta}.
	\end{equation}
	Fractional integrals of  $\psi(t)$, for different values of $\alpha$, $\beta$, $\sigma$ and $\rho$ are graphically illustrated in Fig. \ref{Fig1}.
	Now we compute Riemann-Liouville substantial derivative of  $\psi(t)=(t^\rho-a^\rho)^{\beta}e^{-\sigma t^\rho}$. Note that 
	\begin{equation}\label{EX2}
	D^{1,\rho}(t^\rho-a^\rho)^\beta=
	\frac{t^{1-\rho}}{\rho}\frac{d}{dt}(t^\rho-a^\rho)^\beta
	=\beta(t^\rho-a^\rho)^{\beta-1}.
	\end{equation}
	Therefore, from definition  of Riemann-Liouville substantial derivative, Lemma \eqref{Lem1} and equation \eqref{EX2} we have
	\begin{align*}
	{}_\sigma D_a^{\alpha,\rho}\psi(t)&={}_\sigma D^{1,\rho}{}_\sigma I_a^{1-\alpha,\rho}\left[e^{-\sigma t^{\rho}}(t^\rho-a^\rho)^\beta\right]
	=\frac{\Gamma(\beta+1)}{\Gamma(\beta-\alpha+2)}{}_\sigma D^{1,\rho}\left[e^{-\sigma t^{\rho}}(t^\rho-a^\rho)^{\beta-\alpha+1}\right]\\
	&=\frac{\Gamma(\beta+1)e^{-\sigma t^{\rho}}}{\Gamma(\beta-\alpha+2)}D^{1,\rho}(t^\rho-a^\rho)^{\beta-\alpha+1}
	=\frac{\Gamma(\beta+1)e^{-\sigma t^{\rho}}}{\Gamma(\beta-\alpha+1)}(t^\rho-a^\rho)^{\beta-\alpha}.
	\end{align*}
	Similarly, Caputo type substantial derivative of  $\psi(t)=(t^\rho-a^\rho)^{\beta}e^{-\sigma t^\rho}$ can be computed as
	\begin{align*}
	{}^c_\sigma D_a^{\alpha,\rho}\psi(t)&
	={}_\sigma I_a^{1-\alpha,\rho}{}_\sigma D^{1,\rho}\left[e^{-\sigma t^{\rho}}(t^\rho-a^\rho)^\beta\right]
	={}_\sigma I_a^{1-\alpha,\rho}\left[ e^{-\sigma t^{\rho}}D^{1,\rho}((t^\rho-a^\rho)^\beta)\right]\\
	&={}_\sigma I_a^{1-\alpha,\rho} e^{-\sigma t^{\rho}}D^{1,\rho}(t^\rho-a^\rho)^\beta
	=e^{-\sigma t^{\rho}}I_a^{1-\alpha,\rho} D^{1,\rho}(t^\rho-a^\rho)^\beta\\
	&=\beta e^{-\sigma t^{\rho}}I_a^{1-\alpha,\rho} (t^\rho-a^\rho)^{\beta-1}
	=\frac{\Gamma(\beta+1)e^{-\sigma t^{\rho}}}{\Gamma(\beta-\alpha+1)}(t^\rho-a^\rho)^{\beta-\alpha}.
	\end{align*}
\end{example}


\section{Existence and uniqueness of solutions}

When it comes to the problem of solving a fractional differential equation, the existence and uniqueness results have their own importance. It is necessary to notice in advance whether there is a solution to a given fractional differential equation.  With this in view, here we prove the equivalence between initial value problem (IVP) and Volterra equation. Then, using this equivalence along-with Weissinger’s fixed point theorem, we prove the existence and uniqueness of solution for the following IVP
\begin{align}
{}_\sigma^c{D}_{0}^{\alpha,\rho}\psi(t)&=f(t,\psi(t)),\ \ \ t>0, \label{de}  \\
{}_{\sigma}D_{}^{k,\rho}\psi(0)&=b_k,   \ \ k \in \left\{0,1,2,...,m-1 \right\}, \label{ic}
\end{align}
where $ \sigma>0 $, $ \rho>0$, $\alpha>0 $, $ m= \lceil \alpha \rceil $, $ 	{}_\sigma^c{D}_{0}^{\alpha,\rho} $ is the generalized Caputo-type substantial fractional derivative and $ f: \R^{+} \times \R \to \R.$

For $ K>0 $, $ h^{*}>0 $ and $ b_1,...,b_m \in \R $, define the set
$$
H:=  \left\{ (t,\psi(t)):  0 \leq t \leq h^{*}, \left| \psi(t)- e^{-\sigma t^\rho}  \sum_{k=0}^{m-1}\frac{b_k}{\Gamma(k+1)}  t^{\rho k} \right| \leq K  \right\}.	$$
Following will be assumed while establishing the subsequent results.

\begin{itemize}
	\item[(H1)] $ f:H \to \R $ is both continuous and bounded in $ H $;
	
	\item[(H2)] $ f $  satisfies the Lipschitz condition with respect to the second variable, i.e. for some constant $ L>0 $ and for all $ (t,\psi(t)),(t,\tilde{\psi}(t)) \in  H $, we have
	
	$$   
	|f(t,\psi(t))-f(t,\tilde{\psi}(t))|\leq L|\psi(t)-\tilde{\psi}(t)|.
	$$
	
\end{itemize}

For convenience, we introduce some notations. Let
$
h:= \min \left\{ h^{*} , \tilde{h} , \Big(\frac{\Gamma(\alpha+1)K}{M} \Big)^{\frac{1}{\rho \alpha}} \right\}  	   
$
where $ M:=  \sup_{(x,y )\in H} |f(x,y)| $ and $ \tilde{h} $ being a positive real number, fulfills the inequality
$
\tilde{h}< \Big(\frac{\Gamma(\alpha+1)}{L} \Big)^{\frac{1}{\rho \alpha}}.
$ These notations appear frequently in this section. The main results of this section are the generalizations of existence and uniqueness results presented in \cite{U.N.,Morgado,N. J.}.

\begin{theorem} \label{equivalence}
	Assume that $ h>0 $ and $ f: \R^{+} \times \R \to \R $ is continuous. Then $ \psi \in C[0,h] $ is the solution of IVP $ \left(\ref{de}\right)$--$ \left(\ref{ic}\right) $
	if and only if $ \psi \in C[0,h] $ satisfies the Volterra equation
	$$  \psi(t)=e^{-\sigma t^\rho}  \sum_{k=0}^{m-1}\frac{b_k  }{\Gamma(k+1)}t^{\rho k}+\frac{\rho}{\Gamma(\alpha)}\int_0^t\frac{
		e^{-\sigma(t^\rho-s^\rho)}}{(t^\rho-s^\rho)^{1-\alpha}} s^{\rho-1} f(s,\psi(s)) ds.	     $$
\end{theorem}
\begin{proof}
	Let $ \psi\in C[0,h]  $ be a solution of Volterra equation
	$$  \psi(t)=e^{-\sigma t^\rho}  \sum_{k=0}^{m-1}\frac{b_k  }{\Gamma(k+1)}t^{\rho k}+{}_{\sigma}I_{0}^{\alpha,\rho} f(t,\psi(t)). 	     $$
	Apply $	{}_\sigma^c{D}_{0}^{\alpha,\rho} $ to both sides of the above equation. Using Theorem \ref{theorem3} and Example \ref{exp}, we get
	\begin{align*}
	{}_\sigma^c{D}_{0}^{\alpha,\rho}\psi(t) &=  \sum_{k=0}^{m-1}\frac{b_k  }{\Gamma(k+1)}{}_\sigma^c{D}_{0}^{\alpha,\rho}e^{-\sigma t^\rho} t^{\rho k}+{}_\sigma^c{D}_{0}^{\alpha,\rho}{}_{\sigma}I_{0}^{\alpha,\rho} f(t,\psi(t)) \\&= f(t,\psi(t)).
	\end{align*}
	Now we apply $ {}_{\sigma}D_{}^{k,\rho} $ to both sides of Volterra equation, where $ 0 \leq k \leq m-1 $. Using Theorem \ref{theorem3}, Theorem \ref{theorem5} and Example \ref{exp}, we have
	\begin{align*}
	{}_{\sigma}D_{0}^{k,\rho}\psi(t) &=  \sum_{j=0}^{m-1}\frac{b_j  }{\Gamma(j+1)} {}_{\sigma}D_{}^{k,\rho} e^{-\sigma t^\rho}t^{\rho j}+ {}_{\sigma}D_{}^{k,\rho} {}_{\sigma}I_{0}^{\alpha,\rho} f(t,\psi(t))\\&=\sum_{j=0}^{m-1}\frac{b_j  }{\Gamma(j+1)} \Big(\frac{\Gamma(j+1)}{\Gamma(j-k+1)}    e^{-\sigma t^\rho}t^{\rho(j-k)} \Big)+ {}_{\sigma}D_{}^{k,\rho} {}_{\sigma}I_{}^{k,\rho} {}_{\sigma}I_{0}^{\alpha-k,\rho} f(t,\psi(t))\\&=   e^{-\sigma t^\rho} \sum_{j=0}^{m-1}\frac{b_j }{\Gamma(j-k+1)}  t^{\rho(j-k)}+ \frac{\rho}{\Gamma(\alpha-k)}\int_0^t\frac{
		e^{-\sigma(t^\rho-s^\rho)}}{(t^\rho-s^\rho)^{1-\alpha+k}} s^{\rho-1} f(s,\psi(s)) ds.
	\end{align*}
	Clearly for $ j<k $, the summands become identically zero because reciprocal of Gamma function for non-positive integers, vanishes. Furthermore, for $ k<j $, the summands vanish if $ t=0 $. Since $\alpha-k $ is a positive real number, so the integral also vanishes when $ t=0 $. Thus, we are left with the case $ j=k $.
	$$
	{}_{\sigma}D_{}^{k,\rho}\psi(0) = \frac{b_k }{\Gamma(k-k+1)}    e^{-\sigma t^\rho}t^{\rho(k-k)} \Big |_{t=0} =b_k.
	$$

	Conversely, assume that $ \psi \in C[0,h]  $ is the solution of the given IVP. Applying $ {}_{\sigma}I_{0}^{\alpha,\rho}  $ to both sides of the fractional differential equation  \eqref{de}, using the initial conditions $ \eqref{ic} $ and result of Theorem  \ref{caputoid} , we get Volterra equation.
\end{proof}

\begin{theorem}\label{uniquesolutionexisits} Assume that $ f $ satisfies $ (H1) $ and $ (H2) $. Then, Volterra equation $$  \psi(t)=e^{-\sigma t^\rho}  \sum_{k=0}^{m-1}\frac{b_k  }{\Gamma(k+1)}t^{\rho k}+\frac{\rho}{\Gamma(\alpha)}\int_0^t\frac{
		e^{-\sigma(t^\rho-s^\rho)}}{(t^\rho-s^\rho)^{1-\alpha}} s^{\rho-1} f(s,\psi(s)) ds     $$
	possesses a uniquely determined solution $ \psi \in [0,h] $.
\end{theorem}
\begin{proof}
	Define a set 
	
	$$  
	B:=  \left\{  \psi \in C[0,h] : \sup_{0 \leq t \leq h}   \Big| \psi(t) - e^{-\sigma t^\rho}  \sum_{k=0}^{m-1}\frac{b_k  }{\Gamma(k+1)} t^{\rho k} \Big| \leq K   \right\}
	$$
	equipped with the norm $ ||.||_{B} $ 
	$$
	||\psi||_{B} := \sup_{0 < t \leq h} \Big|  \psi(t)  \Big|.
	$$
	It can be seen that $ (B,||.||_{B}) $ is a Banach space.
	Define the operator $ E $ by
	
	$$  E\psi(t):=e^{-\sigma t^\rho}  \sum_{k=0}^{m-1}\frac{b_k  }{\Gamma(k+1)}t^{\rho k}+\frac{\rho}{\Gamma(\alpha)}\int_0^t\frac{
		e^{-\sigma(t^\rho-s^\rho)}}{(t^\rho-s^\rho)^{1-\alpha}} s^{\rho-1} f(s,\psi(s)) ds.	     $$
	It is easy to check that $ E\psi $ is a continuous on interval $ [0,h] $ for $ \psi \in B $. Moreover,
	
	\begin{align*}
	\Big| E\psi(t)- e^{-\sigma t^\rho}  \sum_{k=0}^{m-1}\frac{b_k }{\Gamma(k+1)}  t^{\rho k} \Big|&=\Big|\frac{\rho  }{\Gamma(\alpha)}\int_0^t\frac{
		e^{-\sigma(t^\rho-s^\rho)}}{(t^\rho-s^\rho)^{1-\alpha}} s^{\rho-1} f(s,\psi(s)) ds  \Big| \\& \leq \frac{\rho }{\Gamma(\alpha)} M \int_0^t\frac{s^{\rho-1}
	}{(t^\rho-s^\rho)^{1-\alpha}}  ds \\& = \frac{\rho  }{\Gamma(\alpha)} M \frac{t^{\rho \alpha}}{\rho \alpha}=  \frac{ M }{\Gamma(\alpha+1)} t^{\rho \alpha}\leq K 
	\end{align*}
	for $ t \in [0,h] $, the last step follows from the definition of $ h $. This means that $ E\psi \in B $ for $ \psi \in B $, i.e. $ E $ is the self-map.

	From the definition of operator $ E $ and Volterra equation, it follows that fixed points of $ E $ are solutions of Volterra equation.

	We use Weissinger’s fixed point theorem to prove that the operator $ E $ has a unique fixed point. For $  \psi_1,\psi_2 \in B$, first we will show the following inequality
	
	$$  
	||E^{j}\psi_1-E^{j}\psi_2||_{B} \leq \Big( \frac{  L  h^{\rho \alpha}}{\Gamma(\alpha+1)} \Big)^{j} ||\psi_1-\psi_2||_{B}.
	$$
	Clearly, the above inequality is true for the case $ j=0 $. Assume that it is true for $ j=k-1 $. For $ j=k  $, we have
	\begin{align*}
	\allowdisplaybreaks 
	||E^{k}\psi_1-E^{k}\psi_2||_{{B}} &=\sup_{0 \leq t \leq h} \Big|   E^{k}\psi_1(t)-E^{k}\psi_2(t)   \Big|\\&=  \sup_{0 \leq t \leq h} \Big| E E^{k-1}\psi_1(t)-EE^{k-1}\psi_2(t)   \Big|\\&= \sup_{0 \leq t \leq h} \frac{1}{\Gamma(\alpha)}\Big| \rho \int_0^t\frac{
		e^{-\sigma(t^\rho-s^\rho)}}{(t^\rho-s^\rho)^{1-\alpha}} s^{\rho-1} \Big[f(s,E^{k-1}\psi_1(s))-f(s,E^{k-1}\psi_2(s)) \Big] ds \Big|	\\& \leq  \sup_{0 \leq t \leq h}  \frac{  L }{\Gamma(\alpha)}  	\left\{ \rho \int_0^t\frac{
		1}{(t^\rho-s^\rho)^{1-\alpha}} s^{\rho-1}ds \right\}  
	||E^{k-1}\psi_1-E^{k-1}\psi_2||_{{B}} 	\\& =  \frac{  L }{\Gamma(\alpha)}  	\left\{ \frac{
		h^{\rho \alpha}}{\alpha}  \right\}  
	||E^{k-1}\psi_1-E^{k-1}\psi_2||_{{B}}  = \Big( \frac{  L  h^{\rho \alpha}}{\Gamma(\alpha+1)} \Big)||E^{k-1}\psi_1-E^{k-1}\psi_2||_{{B}} \\&=\Big( \frac{  L  h^{\rho \alpha}}{\Gamma(\alpha+1)} \Big)^{k}||\psi_1-\psi_2||_{{B}}.
	\end{align*}
	Since $ h \leq \tilde{h} $, we have $ \Big( \frac{  L  h^{\rho \alpha}}{\Gamma(\alpha+1)} \Big) < 1. $ Thus, the series $ \sum_{j=0}^{\infty} \Big( \frac{  L  h^{\rho \alpha}}{\Gamma(\alpha+1)} \Big)^{j} $ is convergent.
	This completes the proof.
\end{proof}

Following is an example for which a general method to determine the analytical solution is not available, but Theorem \ref{equivalence} and Theorem \ref{uniquesolutionexisits} allows us to comment on the existence of its unique solution.
\begin{example}
	Consider the IVP
	\begin{align}
	{}_1^c{D}_{0}^{0.5,2}\psi(t)&=t e^{-t^2}\frac{(\psi(t))^2}{1+(\psi(t))^2}, \label{existenceexp}\\
	\psi(0)&=b_0. \label{existenceic} 
	\end{align}
	It can easily be verified that $ f(t,\psi(t))= t e^{-t^2}\frac{(\psi(t))^2}{1+(\psi(t))^2}$ is both, continuous and bounded in $ H $. Furthermore, we show that $ f $ satisfies the Lipschitz condition
	\begin{align*}
	|f(t,\psi(t))-f(t,\tilde{\psi}(t))|&=\left| t e^{-t^2}\frac{(\psi(t))^2}{1+(\psi(t))^2}-t e^{-t^2}\frac{(\tilde{\psi}(t))^2}{1+(\tilde{\psi}(t))^2}\right|\\&=\left|t e^{-t^2}\right|\left| \frac{(\tilde{\psi}(t))^2-(\psi(t))^2}{(1+(\psi(t))^2)(1+(\tilde{\psi}(t))^2)}\right|.
	\end{align*}
	Since $ 1+(\psi(t))^2 \geq 1 $ and $ 1+(\tilde{\psi}(t))^2 \geq 1 $, so
	\begin{align*}
	|f(t,\psi(t))-f(t,\tilde{\psi}(t))| & \leq \Big\{ \sup_{0 \leq t \leq h} \Big| t e^{-t^2} (\tilde{\psi}(t)+\psi(t)) \Big| \Big\} |\tilde{\psi}(t)-\psi(t)| \\& \leq h \Big\{  \Big|\sup_{0 \leq t \leq h}|\tilde{\psi}(t)|+\sup_{0 \leq t \leq h}|\psi(t) | \Big\} |\tilde{\psi}(t)-\psi(t)| \\&= h(K_1+K_2)|\psi(t)-\tilde{\psi}(t)|,
	\end{align*}
	where $ L:= h(K_1+K_2)$ is the Lipschitz constant. Thus, hypothesis $ (H1) $ and $ (H2) $ hold. From Theorem \ref{equivalence} and Theorem \ref{uniquesolutionexisits}, we can deduce that there exists a unique solution of IVP \eqref{existenceexp}-\eqref{existenceic}.
\end{example}

\section{Continuous dependence of solutions on the given data}

In this section, first we prove a Gronwall-type inequality which is the generalized version of Gronwall-type inequalities presented in \cite{Ye,Gong,R.}. Undoubtedly, this inequality plays an important role in the qualitative theory of integral and differential equations. Furthermore, we analyze the continuous dependence of solution of a fractional differential equation on the given data.
\begin{theorem} \label{gronwall}
	Assume that $ p $ and $ q $ are non-negative integrable functions and $ g $ is non-negative and non-decreasing continuous function on $ [a,b] $.
	
	If 
	$$
	p(t) \leq q(t) + {\rho}^{1-\alpha} g(t) \int_a^t\frac{s^{\rho-1}
		e^{-\sigma(t^\rho-s^\rho)}}{(t^\rho-s^\rho)^{1-\alpha}}p(s) ds, \ \ \forall t \in [a,b],
	$$
	
	then
	$$
	p(t)  \leq q(t) +  \int_a^t \sum_{k=1}^{\infty} \frac{ {\rho}^{1-k\alpha}[g(t)\Gamma(\alpha)]^{k}}{\Gamma(k\alpha)} 	e^{-\sigma(t^\rho-s^\rho)}(t^\rho-s^\rho)^{k\alpha - 1} s^{\rho-1} q(s) ds, \ \ \forall \ t \in [a,b].
	$$
	
	Moreover, if $ q $ is non-decreasing, then 
	
	$$
	p(t)  \leq q(t) E_\alpha \Big[ g(t) \Gamma(\alpha) \Big(\frac{(t^\rho-a^\rho)}{\rho}\Big)^{\alpha} \Big], \ \ \forall \ t \in [a,b].
	$$
\end{theorem}
\begin{proof}
	Define operator $ A $ as 
	
	$$
	A \psi(t):= {\rho}^{1-\alpha} g(t) \int_a^t\frac{s^{\rho-1}
		e^{-\sigma(t^\rho-s^\rho)}}{(t^\rho-s^\rho)^{1-\alpha}}\psi(s) ds.
	$$ Then,
	
	$$
	p(t)  \leq q(t)+  A p(t).
	$$
	Iterating successively, for $  n \in \N$, we obtain
	$$
	p(t)  \leq \sum_{k=0}^{n-1} A^{k} q(t)+   A^{n} p(t).
	$$
	By mathematical induction, we show that if $ \psi $ is non-negative, then, 
	
	$$
	A^{k} \psi(t) \leq {\rho}^{1-k\alpha} \int_a^t  \frac{ [g(t)\Gamma(\alpha)]^{k}}{\Gamma(k\alpha)} 	e^{-\sigma(t^\rho-s^\rho)}(t^\rho-s^\rho)^{k\alpha - 1} s^{\rho-1} \psi(s) ds. 
	$$
	For $ k=1 $, the equality holds. Assume that it is true for $ k \in \N $. Then,  
	\begin{align*}
	A^{k+1} \psi(t)&= A(A^{k} \psi(t)) 	\\& \leq  A \Big( {\rho}^{1-k\alpha} \int_a^\tau  \frac{ [g(\tau)\Gamma(\alpha)]^{k}}{\Gamma(k\alpha)} 	e^{-\sigma(\tau^\rho-s^\rho)}(\tau^\rho-s^\rho)^{k\alpha - 1} s^{\rho-1} \psi(s) ds \Big) \\&=  {\rho}^{1-\alpha} g(t) \int_a^t\frac{\tau^{\rho-1}
		e^{-\sigma(t^\rho-\tau^\rho)}}{(t^\rho-\tau^\rho)^{1-\alpha}} \Big( {\rho}^{1-k\alpha} \int_a^\tau  \frac{ [g(\tau)\Gamma(\alpha)]^{k}}{\Gamma(k\alpha)} 	e^{-\sigma(\tau^\rho-s^\rho)}(\tau^\rho-s^\rho)^{k\alpha - 1} s^{\rho-1} \psi(s) ds \Big) d\tau.
	\end{align*}
	By assumption, $ g $ is non-decreasing, so $  g (\tau) \leq g(t) $, $ \forall\ \tau \leq t $. Thus, we have
	$$
	A^{k+1} \psi(t) \leq \frac{ (\Gamma(\alpha))^{k}}{\Gamma(k\alpha)}  {\rho}^{2-(k+1)\alpha} (g(t))^{k+1} \int_a^t \int_a^\tau  	e^{-\sigma(t^\rho-s^\rho)} 
	(t^\rho-\tau^\rho)^{\alpha-1} \tau^{\rho-1}	(\tau^\rho-s^\rho)^{k\alpha - 1} s^{\rho-1} \psi(s) ds  d\tau.
	$$
	Using Fubini's Theorem and Dirichlet's technique, we get
	\begin{align*}
	A^{k+1} \psi(t) & \leq \frac{ (\Gamma(\alpha))^{k}}{\Gamma(k\alpha)}  {\rho}^{2-(k+1)\alpha} (g(t))^{k+1} \int_a^t e^{-\sigma(t^\rho-s^\rho)} s^{\rho-1} \psi(s) \int_s^t   
	(t^\rho-\tau^\rho)^{\alpha-1} \tau^{\rho-1}	(\tau^\rho-s^\rho)^{k\alpha - 1}  d\tau  ds \\&= \frac{ (\Gamma(\alpha))^{k}}{\Gamma(k\alpha)}  {\rho}^{2-(k+1)\alpha} (g(t))^{k+1} \int_a^t e^{-\sigma(t^\rho-s^\rho)} s^{\rho-1} \psi(s) \Big( \frac{\Gamma(\alpha)\Gamma(k\alpha)}{\rho \Gamma(k\alpha+\alpha)}(t^\rho-s^\rho)^{(k+1)\alpha-1} 
	\Big ) ds \\&={\rho}^{1-(k+1)\alpha} \int_a^t  \frac{ [g(t)\Gamma(\alpha)]^{(k+1)}}{\Gamma((k+1)\alpha)} 	e^{-\sigma(t^\rho-s^\rho)}(t^\rho-s^\rho)^{(k+1)\alpha - 1} s^{\rho-1} \psi(s) ds.
	\end{align*}
	Now we prove that $ A^{n}p(t) \to 0 $ as $n \to \infty$. Since $ g $ is continuous  on $ [a, b] $, so by Weierstrass theorem, $ \exists $ a constant $ M > 0 $ such that $ g (t) \leq M, \ \ \forall \ t \in [a,b]. $
	$$
	\implies A^{n}p(t) \leq {\rho}^{1-n\alpha} \int_a^t  \frac{ [M\Gamma(\alpha)]^{n}}{\Gamma(n\alpha)} 	e^{-\sigma(t^\rho-s^\rho)}(t^\rho-s^\rho)^{n\alpha - 1} s^{\rho-1} p(s) ds.
	$$
	Consider the series 
	$$
	\sum_{n=1}^{\infty}  \frac{ [M\Gamma(\alpha)]^{n}}{\Gamma(n\alpha)}.
	$$
	Using the relation $$ \lim_{n \to \infty} \frac{\Gamma(n \alpha)(n \alpha)^{\alpha}}{\Gamma(n \alpha +\alpha)} = 1, $$ and ratio test, we deduce that the series converges and therefore $ A^{n}p(t) \to 0 $ as $n \to \infty$. Thus,
	\begin{align*}
	p(t) & \leq \sum_{k=0}^{\infty} A^{k} q(t) \\ & \leq  q(t) +  \int_a^t \sum_{k=1}^{\infty} \frac{ {\rho}^{1-k\alpha}[g(t)\Gamma(\alpha)]^{k}}{\Gamma(k\alpha)} 	e^{-\sigma(t^\rho-s^\rho)}(t^\rho-s^\rho)^{k\alpha - 1} s^{\rho-1} q(s) ds.
	\end{align*}
	Additionally, if $ q $ is non-decreasing, then, $ q(s) \leq q(t) $, $ \forall \ s \in [a,t]  $. So,
	\begin{align*}
	p(t) & \leq  q(t) \Big[1 +  \sum_{k=1}^{\infty} \frac{ {\rho}^{1-k\alpha}[g(t)\Gamma(\alpha)]^{k}}{\Gamma(k\alpha)}   \int_a^t	e^{-\sigma(t^\rho-s^\rho)}(t^\rho-s^\rho)^{k\alpha - 1} s^{\rho-1}  ds \Big] \\& \leq  q(t) \Big[1 +  \sum_{k=1}^{\infty} \frac{ {\rho}^{1-k\alpha}[g(t)\Gamma(\alpha)]^{k}}{\Gamma(k\alpha)}   \int_a^t	(t^\rho-s^\rho)^{k\alpha - 1} s^{\rho-1}  ds \Big] \\&=  q(t) \Big[1 +  \sum_{k=1}^{\infty} \frac{ {\rho}^{-k\alpha}[g(t)\Gamma(\alpha)(t^\rho-a^\rho)^{\alpha }]^{k}}{\Gamma(k\alpha+1)}  \Big]  \\&=   q(t) E_\alpha \Big[ g(t) \Gamma(\alpha) \Big(\frac{(t^\rho-a^\rho)}{\rho}\Big)^{\alpha} \Big].
	\end{align*}
	
\end{proof}

Next we look at the dependence of solution of a fractional differential equation on the initial values.
\begin{theorem}\label{dependenceonic}
	Assume that $ \psi $ is the solution of the IVP $ \eqref{de}-\eqref{ic}  $ and $ \phi $ is the solution of the following IVP
	\begin{align}
	{}_\sigma^c{D}_{0}^{\alpha,\rho}\phi(t)&=f(t,\phi(t)),\ \ \ t>0, \label{zde}  \\
	{}_{\sigma}D_{}^{k,\rho}\phi(0)&=c_k,     \ \ k \in \left\{0,1,2,...,m-1 \right\}. \label{zic}
	\end{align}
	Let $ \epsilon := \max_{k=0,1,...,m-1} |b_k-c_k| $. If $ \epsilon $ is sufficiently small, then $ \exists$ some constant $ h>0 $ such that $ \psi $ and $ \phi $ are defined on $ [0,h] $, and 
	$$
	\sup_{0 \leq t \leq h} |\psi(t)-\phi(t)|=\mathcal{O} (\epsilon).
	$$
\end{theorem}

\begin{proof}
	
	Let $ \psi $ and $ \phi $ be defined on $ [0,h_1] $ and $ [0,h_2] $, respectively. Take $ h=\min \left\{h_1,h_2\right\} $, then both the functions $ \psi $ and $ \phi $, are at-least defined on interval $[0,h] $. Define $ \delta(t):= \psi(t)-\phi(t) $, then $ \delta $ is the solution of the following IVP
	\begin{align}
	{}_\sigma^c{D}_{0}^{\alpha,\rho}\delta(t)&=f(t,\psi(t))-f(t,\phi(t)),\ \ \ t>0, \label{diffeq}  \\
	{}_{\sigma}D_{}^{k,\rho}\delta(0)&=b_k-c_k,     \ \ k \in \left\{0,1,2,...,m-1 \right\}.  \label{inicon} 
	\end{align}
	The IVP $ \eqref{diffeq}-\eqref{inicon} $ is equivalent to Volterra equation 
	$$  \delta(t)=e^{-\sigma t^\rho}  \sum_{k=0}^{m-1}\frac{(b_k-c_k)  }{\Gamma(k+1)}t^{\rho k}+\frac{\rho}{\Gamma(\alpha)}\int_0^t\frac{
		e^{-\sigma(t^\rho-s^\rho)}}{(t^\rho-s^\rho)^{1-\alpha}} s^{\rho-1} \Big(f(s,\psi(s))-f(s,\phi(s))\Big) ds.	     $$
	Taking absolute of above equation and using triangle inequality and Lipschitz condition on $ f $, we get
	\begin{align*}
	|\delta(t)|&=\Big| e^{-\sigma t^\rho}  \sum_{k=0}^{m-1}\frac{(b_k-c_k)  }{\Gamma(k+1)}t^{\rho k}+\frac{\rho}{\Gamma(\alpha)}\int_0^t\frac{
		e^{-\sigma(t^\rho-s^\rho)}}{(t^\rho-s^\rho)^{1-\alpha}} s^{\rho-1} \Big(f(s,\psi(s))-f(s,\phi(s))\Big) ds\Big| \\& \leq \Big| e^{-\sigma t^\rho} \Big|  \sum_{k=0}^{m-1}\frac{ t^{\rho k} }{\Gamma(k+1)}\Big|b_k-c_k\Big|+\frac{\rho}{\Gamma(\alpha)}\int_0^t\frac{
		e^{-\sigma(t^\rho-s^\rho)}}{(t^\rho-s^\rho)^{1-\alpha}} s^{\rho-1} \Big|f(s,\psi(s))-f(s,\phi(s))\Big| ds  \\& \leq   \sum_{k=0}^{m-1}\frac{ h^{\rho k} }{\Gamma(k+1)}\max_{k=0,1,...,m-1} \Big|b_k-c_k\Big|+\frac{\rho L}{\Gamma(\alpha)}\int_0^t\frac{
		e^{-\sigma(t^\rho-s^\rho)}}{(t^\rho-s^\rho)^{1-\alpha}} s^{\rho-1} \Big|\psi(s)-\phi(s)\Big| ds \\&= m\epsilon   \sum_{k=0}^{m-1}\frac{ h^{\rho k} }{\Gamma(k+1)} +\frac{\rho L}{\Gamma(\alpha)}\int_0^t\frac{
		e^{-\sigma(t^\rho-s^\rho)}}{(t^\rho-s^\rho)^{1-\alpha}} s^{\rho-1} |\delta(s)| ds.
	\end{align*}
	Taking $ p(t)=|\delta(t)| $, $ q(t)= m\epsilon   \sum_{k=0}^{m-1}\frac{ h^{\rho k} }{\Gamma(k+1)} $ and $ g(t)=\frac{\rho^{\alpha} L}{\Gamma(\alpha)} $, and using Theorem $ \ref{gronwall} $, we find
	$$
	|\delta(t)| \leq m\epsilon  \sum_{k=0}^{m-1}\frac{ h^{\rho k} }{\Gamma(k+1)} E_{\alpha} ( L t^{\rho \alpha} ) \leq m\epsilon  \sum_{k=0}^{m-1}\frac{ h^{\rho k} }{\Gamma(k+1)} E_{\alpha} ( L h^{\rho \alpha} ) = \mathcal{O} (\epsilon),
	$$
	and this completes the proof.

	Now we discuss an example to verify the statement of Theorem \ref{dependenceonic}.

	\begin{example} \label{stabilityexp}
		The unique analytical solutions of the following four IVPs
		$$
		{}_1^c{D}_{0}^{0.5,0.5}\psi_{i}(t)=0.9\psi_{i}(t), \ \
		\psi_{1}(0)=1,     \ \ \psi_{2}(0)=1.2,     \ \ \psi_{3}(0)=1.4,     \ \ \psi_{4}(0)=1.6,
		$$
		are given by
		$$
		\psi_{i}(t)=\psi_{i}(0) e^{-t^{0.5}} E_{0.5}(0.9t^{0.25}), \ \ 0 \leq t \leq h.
		$$ 
		Plots of these solutions are given in Fig. \ref{Fig2}.
		\begin{figure}[h!]
			\centering
			{\includegraphics[width=0.60\textwidth]{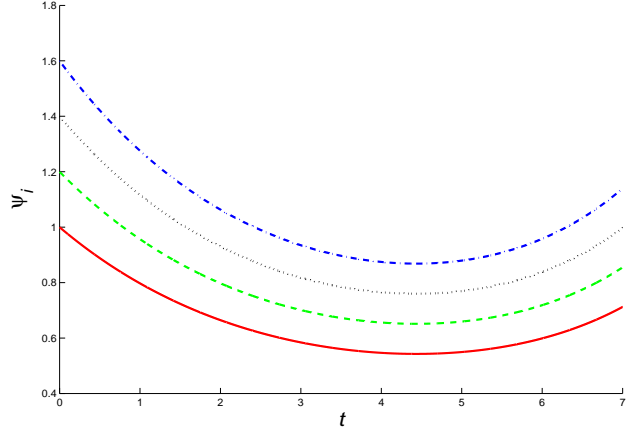}}
			\caption{Graphs of solutions from Example \ref{stabilityexp}.}
			\label{Fig2}
		\end{figure}
		From the Fig. \ref{Fig2}, we can see that change in solutions is bounded by the change in initial conditions on the closed interval $ [0,h]. $ Thus, Example \ref{stabilityexp} verifies the statement of Theorem \ref{dependenceonic}.
	\end{example}	
	
\end{proof}

In the next theorem, we analyze the dependence of solution of the fractional differential equation on the force function $ f $. 
\begin{theorem}
	Assume that $ \psi $ is the solution of the IVP $ \eqref{de}-\eqref{ic}  $ and $ \phi $ is the solution of the following IVP
	\begin{align}
	{}_\sigma^c{D}_{0}^{\alpha,\rho}\phi(t)&=\tilde{f}(t,\phi(t)),\ \ \ t>0, \label{ftildede}  \\
	{}_{\sigma}D_{}^{k,\rho}\phi(0)&=b_k,     \ \ k \in \left\{0,1,2,...,m-1 \right\}, \label{ftildeic}
	\end{align}
	where $ \tilde{f} $ satisfies the same conditions as $f$. Let $ \epsilon := \max_{(t,\phi(t)) \in H} |f(t,\phi(t))-\tilde{f}(t,\phi(t))| $. If $ \epsilon $ is sufficiently small, then $ \exists$ some constant $ h>0 $ such that $ \psi $ and $ \phi $ are defined on $ [0,h] $, and 
	$$
	\sup_{0 \leq t \leq h} |\psi(t)-\phi(t)|=\mathcal{O} (\epsilon).
	$$
\end{theorem}

\begin{proof}
	
	Let $ \psi $ and $ \phi $ be defined on $ [0,h_1] $ and $ [0,h_2] $, respectively. Take $ h=\min \left\{h_1,h_2\right\} $, then both the functions $ \psi $ and $ \phi $, are at-least defined on interval $[0,h] $. Define $ \delta(t):= \psi(t)-\phi(t) $, then $ \delta $ is the solution of the following IVP
	\begin{align}
	{}_\sigma^c{D}_{0}^{\alpha,\rho}\delta(t)&=f(t,\psi(t))-\tilde{f}(t,\phi(t)),\ \ \ t>0, \label{tildedeltadiffeq}  \\
	{}_{\sigma}D_{}^{k,\rho}\delta(0)&=0,     \ \ k \in \left\{0,1,2,...,m-1 \right\}.  \label{tildedeltainicon} 
	\end{align}
	The IVP $ \eqref{tildedeltadiffeq}-\eqref{tildedeltainicon} $ is equivalent to Volterra equation
	$$  \delta(t)=\frac{\rho}{\Gamma(\alpha)}\int_0^t\frac{
		e^{-\sigma(t^\rho-s^\rho)}}{(t^\rho-s^\rho)^{1-\alpha}} s^{\rho-1} \Big(f(s,\psi(s))-\tilde{f}(s,\phi(s))\Big) ds.	     $$
	Taking absolute of above equation and using Lipschitz condition on $ f $, we get
	\begin{align*}
	|\delta(t)|&=\Big| \frac{\rho}{\Gamma(\alpha)}\int_0^t\frac{
		e^{-\sigma(t^\rho-s^\rho)}}{(t^\rho-s^\rho)^{1-\alpha}} s^{\rho-1} \Big[\Big(f(s,\psi(s))-{f}(s,\phi(s))\Big) + \Big(f(s,\phi(s))-\tilde{f}(s,\phi(s))\Big) \Big] ds\Big| \\& \leq \frac{\rho}{\Gamma(\alpha)} \Big\{  \int_0^t\frac{
		e^{-\sigma(t^\rho-s^\rho)}}{(t^\rho-s^\rho)^{1-\alpha}} s^{\rho-1} \Big|f(s,\psi(s))-{f}(s,\phi(s))\Big|   ds + \int_0^t\frac{
		e^{-\sigma(t^\rho-s^\rho)}}{(t^\rho-s^\rho)^{1-\alpha}} s^{\rho-1} \Big| f(s,\phi(s))-\tilde{f}(s,\phi(s))\Big| ds \Big\}  \\& \leq \frac{\rho L}{\Gamma(\alpha)}   \int_0^t\frac{
		e^{-\sigma(t^\rho-s^\rho)}}{(t^\rho-s^\rho)^{1-\alpha}} s^{\rho-1} \Big|\delta(s)\Big|   ds + \frac{\rho }{\Gamma(\alpha)} \int_0^t\frac{
		s^{\rho-1}}{(t^\rho-s^\rho)^{1-\alpha}} \max_{(s,\phi(s)) \in H}  \Big| f(s,\phi(s))-\tilde{f}(s,\phi(s))\Big| ds   \\& \leq \frac{\rho L}{\Gamma(\alpha)}   \int_0^t\frac{
		e^{-\sigma(t^\rho-s^\rho)}}{(t^\rho-s^\rho)^{1-\alpha}} s^{\rho-1} \Big|\delta(s)\Big|   ds + \frac{\epsilon }{\Gamma(\alpha+1)}t^{\rho \alpha} \leq   \frac{\epsilon h^{\rho \alpha} }{\Gamma(\alpha+1)}+\frac{\rho L}{\Gamma(\alpha)}   \int_0^t\frac{
		e^{-\sigma(t^\rho-s^\rho)}}{(t^\rho-s^\rho)^{1-\alpha}} s^{\rho-1} \Big|\delta(s)\Big|   ds
	\end{align*}
	Taking $ p(t)=|\delta(t)| $, $ q(t)= \frac{\epsilon h^{\rho \alpha} }{\Gamma(\alpha+1)} $ and $ g(t)=\frac{\rho^{\alpha} L}{\Gamma(\alpha)} $, and using Theorem $ \ref{gronwall} $, we find
	$$
	|\delta(t)| \leq \frac{\epsilon h^{\rho \alpha} }{\Gamma(\alpha+1)} E_{\alpha} ( L t^{\rho \alpha} ) \leq \frac{\epsilon h^{\rho \alpha} }{\Gamma(\alpha+1)}	 E_{\alpha} ( L h^{\rho \alpha} ) = \mathcal{O} (\epsilon).
	$$
	This completes the proof.
\end{proof}

Finally, we explore the consequences of perturbing the order of the fractional differential equation.

\begin{theorem}
	Assume that $ \psi $ is the solution of the IVP $ \eqref{de}-\eqref{ic}  $ and $ \phi $ is the solution of the following IVP
	\begin{align}
	{}_\sigma^c{D}_{0}^{\tilde{\alpha},\rho}\phi(t)&=f(t,\phi(t)),\ \ \ t>0, \label{alphatildede}  \\
	{}_{\sigma}D_{}^{k,\rho}\phi(0)&=b_k,     \ \ k \in \left\{0,1,2,...,\tilde{m}-1 \right\}, \label{alphatildeic}
	\end{align}
	where $ \tilde{\alpha} > \alpha$ and $ \tilde{m}:=  \lceil \tilde{\alpha} \rceil $. Let $ \epsilon := \tilde{\alpha}-\alpha $ and
	
	\[
	\tilde{\epsilon} :=
	\begin{cases}
	0 & \text{if $\tilde{m}=m,$} \\
	\max \Big\{|b_k|: m \leq k \leq  \tilde{m}-1 \Big\} & \text{otherwise.}
	\end{cases}
	\]
	If $ \epsilon $ and $ \tilde{\epsilon} $ are sufficiently small, then $ \exists$ some constant $ h>0 $ such that $ \psi $ and $ \phi $ are defined on $ [0,h] $, and 
	$$
	\sup_{0 \leq t \leq h} |\psi(t)-\phi(t)|=\mathcal{O} (\epsilon) + \mathcal{O}(\tilde{\epsilon}).
	$$
\end{theorem}
\begin{proof}
	
	Let $ \psi $ and $ \phi $ be defined on $ [0,h_1] $ and $ [0,h_2] $, respectively. Take $ h=\min \left\{h_1,h_2\right\} $, then both the functions $ \psi $ and $ \phi $, are at-least defined on interval $[0,h] $. Define $ \delta(t):= \psi(t)-\phi(t) $, then using Theorem \ref{equivalence}
	\begin{align*}
	\delta(t)&=-e^{-\sigma t^\rho}  \sum_{k=m}^{\tilde{m}-1}\frac{b_k  }{\Gamma(k+1)}t^{\rho k}+\frac{\rho}{\Gamma(\alpha)}\int_0^t\frac{e^{-\sigma(t^\rho-s^\rho)}}{(t^\rho-s^\rho)^{1-\alpha}} s^{\rho-1} f(s,\psi(s)) ds\\
	&\quad -\frac{\rho}{\Gamma(\tilde{\alpha})}\int_0^t\frac{e^{-\sigma(t^\rho-s^\rho)}}{(t^\rho-s^\rho)^{1-\tilde{\alpha}}} s^{\rho-1} f(s,\phi(s)) ds \\&=-e^{-\sigma t^\rho}  \sum_{k=m}^{\tilde{m}-1}\frac{b_k  }{\Gamma(k+1)}t^{\rho k}+\frac{\rho}{\Gamma(\alpha)}\int_0^t\frac{e^{-\sigma(t^\rho-s^\rho)}}{(t^\rho-s^\rho)^{1-\alpha}} s^{\rho-1} \Big(f(s,\psi(s))-f(s,\phi(s))\Big) ds\\
	&\quad + \int_0^t    \Big( \frac{\rho(t^\rho-s^\rho)^{\alpha-1}}{\Gamma(\alpha)}  -\frac{\rho(t^\rho-s^\rho)^{\tilde{\alpha}-1}}{\Gamma(\tilde{\alpha})}  \Big) e^{-\sigma(t^\rho-s^\rho)} s^{\rho-1} f(s,\phi(s)) ds. 
	\end{align*}
	Taking absolute of above equation and using Lipschitz condition on $ f $, we get
	\begin{align*}
	|\delta(t)|& \leq \sum_{k=m}^{\tilde{m}-1}\frac{h^{\rho k} }{\Gamma(k+1)} \Big| b_k \Big| +\frac{\rho L}{\Gamma(\alpha)}\int_0^t\frac{e^{-\sigma(t^\rho-s^\rho)}}{(t^\rho-s^\rho)^{1-\alpha}} s^{\rho-1} |\delta(s)| ds\\
	&\quad +   \max_{(x,y) \in H} \Big|f(x,y)\Big| \int_0^t    \Big| \frac{\rho(t^\rho-s^\rho)^{\alpha-1}}{\Gamma(\alpha)}  -\frac{\rho(t^\rho-s^\rho)^{\tilde{\alpha}-1}}{\Gamma(\tilde{\alpha})}  \Big|  s^{\rho-1} ds \\& \leq  \mathcal{O} (\tilde{\epsilon}) +\frac{\rho L}{\Gamma(\alpha)}\int_0^t\frac{e^{-\sigma(t^\rho-s^\rho)}}{(t^\rho-s^\rho)^{1-\alpha}} s^{\rho-1} |\delta(s)| ds\\
	&\quad +   M \int_0^h    \Big| \frac{(v)^{\alpha-1}}{\Gamma(\alpha)}  -\frac{(v)^{\tilde{\alpha}-1}}{\Gamma(\tilde{\alpha})}  \Big|  dv. 
	\end{align*}
	It can be seen that the zero of above integrand is $ v_0=\Big( \frac{\Gamma(\tilde{\alpha})}{\Gamma({\alpha})} \Big)^\frac{1}{{\tilde{\alpha}-\alpha}} $. If $ h \leq  v_0 $, then absolute value sign can be taken outside the integral. In other case, the interval of integration must be separated at $ v_0 $, and each integral can be evaluated easily. Thus in any case, we find that the integral is bounded by $ \mathcal{O} (\tilde{\alpha}-\alpha)=\mathcal{O} (\epsilon) $. Thus, we have
	$$
	|\delta(t)| \leq  \mathcal{O} (\tilde{\epsilon}) + \mathcal{O} (\epsilon) +\frac{\rho L}{\Gamma(\alpha)}\int_0^t\frac{e^{-\sigma(t^\rho-s^\rho)}}{(t^\rho-s^\rho)^{1-\alpha}} s^{\rho-1} |\delta(s)| ds
	$$
	and using Theorem \ref{gronwall}, the desired result can be obtained.
\end{proof}


\end{document}